\documentclass[11pt,leqno]{article}

\usepackage[margin=1in]{geometry}
\usepackage{amsmath,amssymb,amsthm,mathtools}
\usepackage{enumerate}
\usepackage{graphicx}
\graphicspath{ {./figures/} }
\usepackage{subcaption}
\usepackage{xcolor}
\usepackage{cite}
\usepackage{hyperref}
\usepackage{mathrsfs}
\usepackage{authblk}

\newtheorem{theorem}{Theorem}[section]

\newtheorem{proposition}{Proposition}[section]

\theoremstyle{definition}

\newtheorem{remark}{Remark}[section]

\numberwithin{equation}{section}

\newcommand{\Lip}{\mathrm{Lip}}

\newcommand{\Li}{{L^\infty}}
\newcommand{\Lii}{{L^\infty([0,1])}}
\newcommand{\M}{\mathcal{M}}

\newcommand{\Mloc}{\mathcal{M}_{}^*}
\newcommand{\Nmon}{\mathcal{N}_{\mathsf{mon}}}
\newcommand{\Bezier}{B{\'e}zier }

\newcommand{\norm}[1]{\| {#1}\| }

\newcommand{\set}[1]{\left\{#1\right\}}
\newcommand{\void}{\varnothing}
\newcommand{\abs}[1]{\left|#1\right|}

\newcommand{\R}{\mathbb{R}}
\newcommand{\brac}[1]{\left(#1\right)}

\renewcommand{\epsilon}{\varepsilon}
\newcommand{\eps}{\epsilon}
\newcommand{\eqindent}{\displayindent0pt\displaywidth\textwidth}

\definecolor{blackscomment}{RGB}{189,113,57}

\title{Optimal $C^{1,1}$ and Quasi-Optimal $C^2$ Monotone Interpolation with Curvature Control}
\author[1]{Fushuai Jiang}
\author[2]{Garving K. Luli}
\affil[1]{Department of Mathematics, City University of Hong Kong}
\affil[2]{Department of Mathematics, University of California - Davis}
\date{}

\begin{document}

\maketitle

\begin{abstract}
We study monotone Hermite interpolation on an interval, where both function values and first derivatives are prescribed at the nodes. Among all $C^{1,1}$ interpolants, we seek one with optimal curvature, measured by $\|F''\|_{L^\infty}$. In this paper, we analyze the limitations of some classical techniques, and provide an explicit optimal construction in $C^{1,1}$ given by quadratic splines by studying the optimal velocity profile. 
Moreover, given $E = \{x_1,\cdots,x_N\}$ and $f:E\to \R$ (without derivatives), 
we also provide a formula to compute the corresponding trace seminorm
\[
\inf\Bigl\{
\|F''\|_{L^\infty}
:
F(x)=f(x) \text{ on $E$ and } 
F'\ge 0 \text{ everywhere}
\Bigr\}.
\]
In addition, we also describe how to mollify $C^{1,1}$ solutions to $C^2$ while preserving monotonicity and sacrificing a controlled amount of optimality. 
\end{abstract}

\section{Introduction}
Let $I\subset \R$ be a compact interval. We use $C^{1,1}(I) \cong W^{2,\infty}(I)$ to denote the function space of differentiable functions on $I$ with Lipschitz derivatives, equipped with the seminorm
\begin{equation*}
    \norm{F}_{\dot C^{1,1}(I)} := \norm{F''}_{\Li(I)}
\end{equation*}
where we interpret $F''$ in the sense of distribution. We use $C^2(I)$ to denote the vector space of twice continuously differentiable functions on (the interior of) $I$, equipped with the same seminorm as above.

Let $E = \set{x_i: i = 0, 1, \cdots, N}\subset \R$ with
$x_0 < x_1 < \dots < x_N$. Given Hermite data consisting of a monotone $f:E \to \R$ and its slope $f':E \to  [0,\infty)$, 
our goal is to construct a $C^{1,1}$ or $C^{2}$ Hermite interpolant \(F\) with small curvature measured in $\norm{F''}_{\Li}$. That is, given the first order constraints
\begin{equation}\label{condition1}
F(x) = f(x) \text{ and } F'(x) = f'(x) \text{\quad for all $x \in E$.}
\end{equation}
we seek a function $F: [x_0,x_N]\to \mathbb{R}$ satisfying \eqref{condition1} such that
\begin{equation}\label{condition2}
F'(x)\ge 0 \quad \text{for all } x\in [x_0,x_N],
\end{equation}
and
\begin{equation}
    \norm{F''}_{\Li([x_0,x_N])}
\le
C \cdot \inf \left\{
\norm{G''}_{\Li([x_0,x_N])} :
G \text{ satisfies } \eqref{condition1} \text{ and } \eqref{condition2}
\right\},
\label{eq:quasi-optimal def}
\end{equation}
where \(C\ge 1\) is a universal constant. We call such an $F$ satisfying \eqref{condition1}--\eqref{eq:quasi-optimal def} {\em quasi-optimal}, and if $C = 1$, we call it {\em optimal}. 

In this paper, we will explicitly construct optimal $C^{1,1}$ monotone increasing Hermite interpolants, and quasi-optimal $C^{2}$ monotone Hermite interpolants. Moreover, we will provide an explicit formula to compute the infimum in \eqref{eq:quasi-optimal def}. The monotone decreasing case is equivalent to the monotone increasing case by replacing $F$ with $-F$. 

Since we only work with a finite number of data points and our construction is always piecewise smooth, there is no substantial difference between working with $C^{1,1}$ function and $C^2$ functions, with the understanding that we can always mollify finitely many corners to upgrade a well-behaved $C^{1,1}$ solution to a $C^2$ solution while sacrificing a controlled amount of optimality; see Sections \ref{sect:reduction} and \ref{sect:C2 smoothing} for the smooth patching and local mollification arguments, respectively. From an optimization point of view, it is also convenient to work with $C^{1,1}$ due to its compactness property. As such, we will focus our discussion on $C^{1,1}$ solutions for the rest of the paper.

\subsection{Background and prior work}
Spline interpolation and smooth extension theory have a long history in approximation theory; see, for example, the foundational work of Birkhoff and de Boor and the classical monograph of de Boor \cite{BirkhoffdeBoor1965,deBoor1978}. Classical spline methods provide convenient and powerful tools for interpolation and approximation, but in general, they do not enforce geometric shape constraints such as monotonicity.

The literature on monotone interpolation has been dominated by local piecewise-polynomial constructions, especially piecewise cubic Hermite methods. The seminal work of Fritsch and Carlson introduced monotone piecewise cubic interpolation, and the method of Fritsch and Butland provided a simple local construction of monotone piecewise cubic interpolants \cite{FritschCarlson1980,FritschButland1984}. These methods are computationally effective and remain highly influential, but the resulting interpolants are generally only \(C^1\) with no further global control on the (distributional) second derivative. Related developments include Steffen's monotone interpolation method, Huynh's analysis of accurate monotone cubic interpolation, and Higham's local monotone piecewise cubic interpolation \cite{Steffen1990,Huynh1991,Higham1992}. More recent work includes nonlinear monotone Hermite constructions based on ENO (essentially non-oscillatory) or WENO (weighted essentially non-oscillatory) ideas and optimization-based monotone cubic spline frameworks \cite{ArandigaBaezaYanez2013,WolbergAlfy2002}.

A fundamental difficulty is that monotonicity is a global inequality constraint, while second-order smoothness imposes strong compatibility conditions across adjacent intervals. As a result, ordinary interpolating cubic splines, although naturally $C^2$, may violate monotonicity even for monotone data. Indeed, within the fixed-knot cubic-spline setting, interpolation, global $C^2$ regularity, and monotonicity are not simultaneously compatible in general; see, for example, the discussions in \cite{WolbergAlfy2002,Huynh1991}. Because of this, the classical cubic framework is too rigid to solve the full problem in the $C^{1,1}$ or $C^2$ setting.

To overcome this obstruction, several authors enlarged the admissible class of interpolants. One line of work studies quartic, quintic, or rational splines with additional shape parameters or extra degrees of freedom. In particular, Ulrich and Watson developed positivity conditions for quartic polynomials that underpin monotone higher-order constructions \cite{UlrichWatson1994}, while He\ss{} and Schmidt studied positive quartic and monotone quintic \(C^2\)-spline interpolation \cite{HessSchmidt1994}. Later work constructed monotonicity-preserving \(C^2\) rational cubic splines, unconditionally monotone \(C^2\) quartic spline methods, and monotone quintic spline algorithms \cite{AbbasMajidAli2012,YaoNelson2018,LuxWatsonChangXuWangHong2020}. These results show that \(C^2\) monotone interpolation is achievable in richer spline families, but typically through specialized ans\"atze, auxiliary parameters, or interval-by-interval admissibility conditions.

There is also a broader literature on shape-constrained and monotone regression splines, motivated in part by statistics and data analysis \cite{Ramsay1988,LeitenstorferTutz2007}. However, these methods typically address fitting or smoothing problems rather than exact Hermite interpolation with quantitative control on $F''$.

\subsection{Contribution and significance}
The result established in this paper is of a different nature from the classical shape-preserving spline literature. Rather than searching for a monotone interpolant within a prescribed spline family, we solve the full extension problem in the seminorms $\norm{F''}_{\Li}$. More precisely, we construct a globally smooth function satisfying the Hermite interpolation conditions \eqref{condition1}, preserving monotonicity on the whole interval as in \eqref{condition2}, and doing so with $\norm{F''}_{\Li}$ as small as possible among all admissible monotone interpolants.

This goes substantially beyond earlier monotone interpolation results. Previous methods were primarily designed to optimize locality, visual smoothness, polynomial reproduction, or membership in a specific spline class; see, for example,\cite{FritschCarlson1980,FritschButland1984,Higham1992,WolbergAlfy2002,ArandigaBaezaYanez2013,AbbasMajidAli2012,YaoNelson2018,LuxWatsonChangXuWangHong2020}. By contrast, our theorem gives a quasi-optimal extension result in the precise functional-analytic sense relevant to $C^2$ or $C^{1,1}$ interpolation. In particular, it yields not merely the existence of some monotone smooth interpolant, but one whose curvature is quantitatively near-minimal in the natural seminorm.

To the best of our knowledge, no previous result in the monotone Hermite interpolation literature establishes this combination of exact Hermite interpolation, 
global $C^{1,1}$ or $C^2$ regularity,
monotonicity, and 
optimal control of $\norm{F''}_{\Li}$
at the level of generality considered here. For this reason, our results should be viewed not as another spline recipe, but as a new extension-theoretic result in shape-constrained interpolation with guaranteed optimal quality.

This paper is motivated by the extensive literature on Whitney extension problems, which dates back to Whitney's seminal work in 1934 \cite{W34-1, W34-2}; see \cite{J23-fp,JL20-Alg,JLLL23, JLO22, FJL23} and references therein.

\subsection*{Acknowledgment}

The first author is supported by a grant from
the City University of Hong Kong (Project No. 7200844). The second author is supported in part by NSF grant 2247429, the Chancellor's Fellowship at UC Davis, and the Simons Gift Fund.

\section{Problem statement, main results, and the structure of the paper}

Let $E=\{x_i: i = 0, 1, \dots, N\}\subset \R$ with $x_0 < x_1 < \cdots < x_N$. Let $f:E \to \R$ with $f(x_i)\leq f(x_{i+1})$ for all $i$.
The problem is to construct a function $F\in C^{1,1}([x_0,x_N])$ or $C^2([x_0,x_N])$ such that
\begin{itemize}
    \item (Hermite matching) $F(x) = f(x)$ and $F'(x) = f'(x)$ for all $x \in E$,
    \item (Monotone increasing) $F'(x)\geq 0$ for all $x \in [x_0,x_N]$, and 
    \item (Curvature control) $\norm{F''}_{\Li([x_0,x_N])} \leq C \cdot \inf\norm{G''}_{\Li([x_0,x_N])}$ where $C \geq 1$ is a universal constant and the infimum ranges over all monotone increasing $G \in C^{1,1}([x_0,x_N])$ with $G = f$ and $G' = f'$ on $E$. 
\end{itemize}
Any such function \(F\) will be called a \emph{quasi-optimal monotone Hermite interpolant}. If $C = 1$, then $F$ is called \emph{optimal}.

In Section \ref{sect:reduction}, we will see that the global problem can be reduced to a two-point problem on the unit interval: $f(0) = 0$, $f(1) = c$, $f'(0) = a$, $f'(1) = b$ for some $a,b,c \in [0,\infty)$. As such, we define
\begin{equation}
    \Mloc(a,b,c) := \inf\set{
    \norm{G''}_{\Lii} : 
    \begin{matrix}
        G \in C^{1,1}([0,1])\\
        G(0) = 0,\, G(1) = c,\\
        G'(0) = a,\, G'(1) = b,\\
        G'(x) \geq 0 \quad \forall x \in [0,1]
    \end{matrix}
    }
    \label{eq:Mstar inf def}
\end{equation}
with the convention that $\inf \void = \infty$. Note that if $a = b = c = 0$, then $G\equiv 0$ is an admissible solution for $\Mloc$, so $\Mloc(0,0,0) = 0$. On the other hand, if $c = 0$ and $a+b > 0$, then the Hermite data (along with continuity) contradict monotonicity, so $\Mloc(a,b,c) = \infty$. 

\medskip

Our two main theorems are the following.

\begin{theorem}
\label{thm:main - Mstar}
Let $a,b,c \geq 0$, and let
\[
c_0:=\frac{a^2+b^2}{2(a+b)}.
\]
Here and below, we adopt the convention that $\frac{0}{0}=0$ and $\frac{\eps}{0} = \infty$ for any $\eps > 0$. Then
\begin{equation}
    \M^*(a,b,c)=
\begin{cases}
\displaystyle \frac{a^2+b^2}{2c}
& \text{ if $0 \leq c \leq c_0$}\\[1em]
\displaystyle \abs{2c-a-b}+\sqrt{(2c-a-b)^2+(b-a)^2}
& \text{ if $c \geq c_0$}
\end{cases}
\label{eq:Mstar}
\end{equation}
Note that $\Mloc(a,b,c_0) = a+b$.
\end{theorem}

\begin{theorem}
    \label{thm:explicit complete}
        Assume $c > 0$. Let $M:= \Mloc(a,b,c)$ be as in \eqref{eq:Mstar}.
        \begin{itemize}
            \item Suppose $c \geq \frac{a+b}{2}$. Let $t_0 = \frac{1}{2}+\frac{b-a}{2M}$ and define
            \begin{equation*}
        G(t) = \begin{cases}
        \frac{M}{2}t^2 + at &\text{ if $0 \leq t \leq t_0$}\\
        - \frac{M}{2}t^2 + (M+b)t + \brac{c-b-\frac{M}{2}} &\text{ if $t_0 \leq t \leq 1$}
        \end{cases}.
    \end{equation*}
        \item Suppose $c_0 \leq c \leq \frac{a+b}{2}$. Let $\tilde t_0 = \frac{1}{2}-\frac{b-a}{2M}$ and define
        \begin{equation*}
           G(t) = \begin{cases}
            -\frac{M}{2}t^2 + at &\text{ if $0 \leq t \leq \tilde t_0$}\\
            \frac{M}{2}t^2 + (b-M)t + \brac{c-b+\frac{M}{2}} &\text{ if $\tilde t_0 \leq t \leq 1$}
        \end{cases}. 
        \end{equation*}
        \item Suppose $0 < c < c_0$. Let $\tau_1 = \frac{a}{M}$ and $\tau_2 = 1-\frac{b}{M}$. Define
        \begin{equation*}
            G(t) = \begin{cases}
                -\frac{M}{2}t^2 + at &\text{ if $0\leq t \leq \tau_1$}
                \\
                \frac{a^2}{2M} &\text{ if $\tau_1 \leq t \leq \tau_2$}
                \\
                \frac{M}{2}(t-1)^2 + b(t-1) + \frac{a^2 + b^2}{2M}
                &\text{ if $\tau_2 \leq t \leq 1$}
            \end{cases}.
        \end{equation*}
        \end{itemize}
        Then $G \in C^{1,1}([0,1])$ is monotone increasing, $G(0) = 0$, $G(1) = c$, $G'(0) = a$, $G'(1) = b$, and 
        \begin{equation*}
            \norm{G''}_{\Lii} = \M^*(a,b,c).
        \end{equation*} 
    \end{theorem}

 We postpone the proof of Theorems \ref{thm:main - Mstar} and \ref{thm:explicit complete} to Section \ref{section:velocity based}, but knowing the exact expression of $\M^*(a,b,c)$ now will make the rest of the exposition more transparent. 

In Section \ref{section: classical techniques}, we analyze some classical techniques, including the Whitney extension operator (Theorem \ref{thm:whitney extension operator}) and \Bezier interpolation (Theorems \ref{thm:bezier cubic} and \ref{thm:bernstein generalized}), for handling the two-point monotone interpolation problem, and further study their limitations. 

In Section \ref{section:velocity based}, we prove Theorems \ref{thm:main - Mstar} and \ref{thm:explicit complete} by studying the profile of the velocity function. To get a glimpse of our idea, we will write the curvature control as a function of the given data $a,b,c$.

\medskip

Finally, as a consequence of Theorem \ref{thm:main - Mstar} and the discussion in Section \ref{sect:reduction}, we have the following global theorem, which allows us to compute the $C^{1,1}$ monotone trace seminorm.

\begin{theorem}\label{thm:trace-seminorm}
Let $E=\{x_i: i = 0, 1, \dots, N\}\subset \R$ with $x_0 < x_1 < \cdots < x_N$. Let $f:E \to \R$ with
\begin{equation*}
    f(x_0)\le f(x_1)\le\cdots\le f(x_N).
\end{equation*}
Define the monotone $C^{1,1}$ trace seminorm by
\begin{equation*}
    \norm{f}_{\dot C^{1,1}_{\mathsf{mon}}(E)}
:=
\inf \set{
\norm{F''}_{\Li([x_0,x_N])}
:
F\in C^{1,1}([x_0,x_N]),\ 
F = f \text{ on }E,\ 
F'\ge 0 \text{ on } [x_0,x_N]}
\end{equation*}
For each $i = 0, 1, \dots, N-1$, define 
\begin{equation*}
    h_i := x_{i+1}-x_i
    \text{\quad and \quad }
    s_i:=\frac{f(x_{i+1})-f(x_i)}{h_i}.
\end{equation*}
For a slope vector $d=(d_0,d_1,\dots,d_N) \in \R^{N+1}_+$ (i.e., $d_i\ge 0$ for all $i$), define
\[
\Nmon(f;d)
:=
\max_{0\le i\le N-1}
\frac{\Mloc(d_i,d_{i+1},s_i)}{h_i}
\text{\quad and \quad }
\Nmon(f) := \inf_{d \in \R^{N+1}_+}\Nmon(f;d).
\]
Then
\begin{equation*}
    \|f\|_{\dot C^{1,1}_{\mathsf{mon}}(E)}
= \Nmon(f).
\end{equation*}
\end{theorem}

In principle, given only the values $f(x_i)$ for $i=0, 1,\cdots,N$, one can use the formula in Theorem \ref{thm:trace-seminorm} to determine the optimal slope vectors $d\in \R^{N+1}_+$ for $\Nmon$, and then apply the constructions in this paper to obtain an optimal $C^{1,1}$ monotone interpolant, thereby solving the Whitney extension problem for monotone functions; 
that is, constructing a $C^{1,1}$ monotone function with (quasi-)optimal $\norm{F''}_\Li$ only from the function values $f:E \to \R$ without knowing the derivatives. 
However, we do not present an efficient practical algorithm for carrying out this construction.

\section{Reduction to a 2-point problem on the unit interval}
\label{sect:reduction}

We state the following elementary fact as a proposition and omit the proof. 

\begin{proposition}\label{prop:C11 patching}
    Let $x_1 < x_2 < x_3$, and let 
    \begin{equation*}
        G_- \in C^{1,1}([x_1, x_2]) \text{\quad and \quad} G_+ \in C^{1,1}([x_2, x_3])
    \end{equation*}
    satisfy
    \begin{itemize}
        \item $G_-(x_2) = G_+(x_2)$, $G_-'(x_2-) = G_+'(x_2+)$, and 
        \item $G_-' \geq 0$ on $[x_1, x_2]$ and $G_+'\geq 0$ on $[x_2, x_3]$.
    \end{itemize}
    Then the function defined piecewise by
    \begin{equation*}
        G(x):= \begin{cases}
            G_-(x) &\text{ if } x \in [x_1, x_2]\\
            G_+(x) &\text{ if } x \in [x_2, x_3]
        \end{cases}
    \end{equation*}
    is $C^{1,1}$ and monotone increasing, with 
    \begin{equation*}
        \norm{G''}_\Li([x_1, x_3]) \leq \max \set{
        \norm{G_-}_{\Li([x_1, x_2])},
        \norm{G_+}_{\Li([x_2, x_3])}
        }.
    \end{equation*}
\end{proposition}


Now we rescale each local problem. For each interval \([x_i,x_{i+1}]\), define
\[
h_i := x_{i+1}-x_i>0,
\qquad
s_i := \frac{f_{i+1}-f_i}{h_i},
\]
and endpoint slopes
\[
a_i := d_i,
\qquad
b_i := d_{i+1}.
\]
Consider the affine change of variables \(t = (x-x_i)/h_i \in [0,1]\) and the normalized function
\[
G_i(t) := \frac{F(x_i+h_i t)-f_i}{h_i}.
\]
Then the interpolation constraints on \([x_i,x_{i+1}]\) become
\[
G_i(0)=0,
\qquad
G_i(1)=s_i,
\qquad
G_i'(0)=a_i,
\qquad
G_i'(1)=b_i,
\]
and monotonicity reads as before
\[
G_i'(t)\ge 0
\quad \text{for all } t\in[0,1].
\]
Moreover,
\[
G_i'(t) = F'(x_i+h_i t),
\qquad
G_i''(t) = h_i\, F''(x_i+h_i t).
\]
Hence
\[
\norm{F''}_{\Li([x_i,x_{i+1}])}
=
h_i^{-1}
\norm{G_i''}_\Lii.
\]
Therefore, minimizing the global curvature \(\|F''\|_{\Li([x_0,x_N])}\) reduces to minimizing $\norm{G_i''}_{\Lii}$ for each $i$ and then taking the maximum, since the nodal derivatives are fixed and agree at the nodes.

\medskip

In view of the discussion above, 
it suffices to analyze the normalized two-point problem on $[0,1]$. Given $a,b,c \geq 0$, we seek $G \in C^{1,1}([0,1])$ such that
\begin{equation}\label{eq:abc-data}
G(0)=0,\qquad G(1)=c,\qquad G'(0)=a,\qquad G'(1)=b,\qquad G'(t)\ge 0
\end{equation}
and $\norm{G''}_\Lii$ as small as possible.

\medskip
\medskip

For completeness, we also provide a method for $C^2$ patching, reducing the global $C^2$ problem to a two-point problem.

\begin{theorem}[Local \( C^2 \) patching]
\label{thm:local-C2-patching-exact}
Let \( x_1<x_2<x_3 \), and let
\[
F_-\in C^2([x_1,x_2]),\qquad F_+\in C^2([x_2,x_3])
\]
satisfy
\[
F_-(x_2)=F_+(x_2)=:f_2,
\qquad
F_-'(x_2)=F_+'(x_2)=:d>0,
\]
and
\[
F_-'(x)\ge 0 \ \text{on } [x_1,x_2],
\qquad
F_+'(x)\ge 0 \ \text{on } [x_2,x_3].
\]
Assume also
\[
\|F_-''\|_{\Li([x_1,x_2])}\le M,
\qquad
\|F_+''\|_{\Li([x_2,x_3])}\le M,
\]
with \( M>0 \). Then there exist universal constants \( c_0,C_0>0 \) such that whenever
\[
0<\delta\le \min\Bigl\{\tfrac{x_2-x_1}{2},\tfrac{x_3-x_2}{2},\,c_0\tfrac dM\Bigr\},
\]
there exists \( \widetilde F\in C^2([x_1,x_3]) \) such that
\begin{enumerate}[(A)]
\item \( \widetilde F=F_- \) on \( [x_1,x_2-\delta] \), and \( \widetilde F=F_+ \) on \( [x_2+\delta,x_3] \),
\item \( \widetilde F(x_2)=f_2 \),
\item \( \widetilde F'(x)\ge 0 \) on \( [x_1,x_3] \),
\item \( \|\widetilde F''\|_{\Li([x_1,x_3])}\le C_0 M \).
\end{enumerate}
\end{theorem}


\begin{proof}
Fix $\delta > 0$ as in the hypothesis. 
Set \( a:=x_2-\delta \) and \( b:=x_2+\delta \). Define the endpoint jets
\[
A_0:=F_-(a),\quad A_1:=F_-'(a),\quad A_2:=F_-''(a),
\]
\[
B_0:=F_+(b),\quad B_1:=F_+'(b),\quad B_2:=F_+''(b).
\]
Let \( P \) be the unique polynomial of degree at most $5$ that matches these jets: \( P(a)=A_0 \), \( P'(a)=A_1 \), \( P''(a)=A_2 \), and likewise at \( b \).

By Taylor expansion at \( x_2 \) (using \( \|F_\pm''\|_\infty\le M \) and \( F_\pm'(x_2)=d \)),
\[
|A_1-d|\le M\delta,\quad |B_1-d|\le M\delta,
\]
\[
A_0=f_2-d\delta+r_-,\quad |r_-|\le\tfrac12 M\delta^2,
\]
\[
B_0=f_2+d\delta+r_+,\quad |r_+|\le\tfrac12 M\delta^2.
\]
Hence
\[
|B_0-A_0-2d\delta|\le M\delta^2,\qquad
\Bigl|\tfrac{A_0+B_0}{2}-f_2\Bigr|\le\tfrac12 M\delta^2.
\]

Change variables via \( x=a+2\delta t \) (\( t\in[0,1] \)) and set \( Q(t):=P(a+2\delta t) \), \( L(t):=A_0+(B_0-A_0)t \), \( R(t):=Q(t)-L(t) \). Then \( R(0)=R(1)=0 \). Direct computation of the endpoint jets of \( R \) (using the expansions above) yields
\[
R'(0)=R'(1)=R''(0)=R''(1)=O(M\delta^2).
\]
Since \( R \) is a quintic polynomial vanishing at the endpoints with \( O(M\delta^2) \) jet data, there exists a universal constant \( C_1>0 \) such that
\[
\|R\|_{\Li}+\|R'\|_{\Li}+\|R''\|_{\Li}\le C_1 M\delta^2.
\]
In particular \( |R(1/2)|\le C_1 M\delta^2 \), so
\[
|P(x_2)-f_2|\le\Bigl|\tfrac{A_0+B_0}{2}-f_2\Bigr|+|R(1/2)|\le C M\delta^2
\]
for a (renamed) universal \( C>0 \). 

Let \( E:=f_2-P(x_2) \) with \( |E|\le C M\delta^2 \). Define \( \psi(t):=t^3(1-t)^3 \) (so \( \psi(1/2)=1/64 \) and \( \psi,\psi',\psi'' \) vanish at $0$ and $1$) and set \( \alpha:=64E \). The corrected polynomial
\[
\widetilde{P}(x)=P(x)+\alpha\,\psi\!\left(\tfrac{x-a}{b-a}\right)
\]
satisfies \( \widetilde{P}(x_2)=f_2 \) while preserving the jets of \( P \) at \( a \) and \( b \) up to second order. Define
\[
\widetilde{F}(x)=\begin{cases}
F_-(x) & x\in[x_1,a],\\
\widetilde{P}(x) & x\in[a,b],\\
F_+(x) & x\in[b,x_3].
\end{cases}
\]
We claim that $\widetilde{F}$ is the desired solution. 

It is clear from construction that $\widetilde{F}$ is $C^2$. It remains to check monotonicity and bound the second derivative. Since $\widetilde{F}$ is defined piecewise, it suffices to analyze $\widetilde{P}$. 

We first check the monotonicity of $\widetilde P$. Note that \( Q'(t)=(B_0-A_0)+R'(t)=2d\delta+O(M\delta^2) \). Choose \( c_0>0 \) sufficiently small (depending only on \( C_1 \)) so that \( \delta\le c_0 d/M \) implies \( Q'(t)\ge d\delta \), and therefore,
\[
P'(x)=\frac{Q'(t)}{2\delta}\ge\frac d2
\]
on \( [a,b] \). The correction term satisfies
\[
\bigl|\alpha\,\psi'\!\bigl(\tfrac{x-a}{b-a}\bigr)\frac{1}{b-a}\bigr|\le 32C\norm{\psi'}_\Lii M\delta
\]
For \( c_0 \) small enough that this is at most \( d/4 \), we obtain
\[
\widetilde{P}'(x)\ge\frac d2-\frac d4=\frac d4>0
\]
on \( [a,b] \). Therefore $\widetilde{P}$ is monotone increasing on $[a,b]$.

Now we bound the second derivative of $\widetilde{P}$. Note that \( L''\equiv 0 \), so \( P''(x)=R''(t)/(4\delta^2) \) and
\[
\|P''\|_\infty\le\frac{C_1 M\delta^2}{4\delta^2}=\frac{C_1}4\,M.
\]
The correction term is bounded by $16C\norm{\psi''}_\Lii M$, so \( \|\widetilde{P}''\|_\infty\le C_0 M \) for a universal \( C_0 \). 

\end{proof}

\section{Classical techniques and their limitations}
\label{section: classical techniques}

Fix the following Hermite data and constraint:
\begin{equation}
G(0)=0,\qquad G(1)=c,\qquad G'(0)=a,\qquad G'(1)=b,\qquad G'(t)\ge 0
\label{eq:abc-data-classical section}
\end{equation}
Recall $\Mloc(a,b,c)$ from \eqref{eq:Mstar}. In this section, we analyze some of the classical techniques in an attempt to construct explicit \(C^2\) interpolants \(\widetilde G\) satisfying \eqref{eq:abc-data-classical section} with
\begin{equation*}
    \norm{\widetilde G''}_\Lii
\le C\,\M^*(a,b,c),
\end{equation*}
for a universal constant \(C \geq 1\).

We use several complementary constructions, each adapted to a different portion of the parameter range and each with its own advantages. When
\[
c>\max\{a,b\},
\]
the problem can be handled by the classical Whitney-type extension construction. See Theorem \ref{thm:whitney extension operator}. In the case where $c$ is close to $\frac{a+b}{2}$, say, 
\[
\abs{c - \frac{a+b}{2}}\leq \frac{\abs{b-a}}{4}
\]
we can use a Bézier-type construction, which is especially simple and efficient. See Theorem \ref{thm:bernstein generalized}.

The regimes described above are more or less sharp for these classical techniques, which break down away from these regimes. To address the lapse, we will provide a more comprehensive (and optimal) solution in Section \ref{section:velocity based} by analyzing the velocity profile.

\subsection{Whitney extension operator for $c \geq \max\set{a,b}$}

In this subsection, we assume that
\begin{equation}\label{eq:regime1-assumption}
c\geq\max\{a,b\}.
\end{equation}
We will construct an explicit \(C^2\) monotone Hermite interpolant using a simple Whitney partition of unity.

Define
\begin{equation}\label{eq:phi-def}
\phi(x):=1-3x^2+2x^3,\qquad x\in[0,1].
\end{equation}
Then \(\phi\in C^\infty([0,1])\) and satisfies
\begin{equation*}\label{eq:phi-endpoints}
\phi(0)=1,\quad \phi'(0)=0,\quad \phi(1)=0,\quad \phi'(1)=0.
\end{equation*}
Differentiating \(\phi\), we have
\begin{equation}\label{eq:non-increasing}
\phi'(x)=-6x(1-x)\le 0.
\end{equation}
Therefore, \(\phi\) is non-increasing on \([0,1]\).
Moreover, we have the derivative estimates
\begin{equation}\label{eq:phi-bounds}
\norm{\phi'}_\Lii=\frac32
\text{\quad and \quad}
\norm{\phi''}_\Lii=6.
\end{equation}

\begin{theorem}\label{thm:whitney extension operator}
Let $c \geq \max\set{a,b}$.
Let \(\phi\) be as in \eqref{eq:phi-def} and define
\begin{equation}\label{eq:regime1-F}
F(x):=(ax)\phi(x) + \bigl(1-\phi(x)\bigr)\bigl(b(x-1)+c\bigr),\qquad x\in[0,1].
\end{equation}
Then the following hold:
\begin{enumerate}[(A)]
\item $F\in C^2([0,1])$ and satisfies the Hermite conditions
\[
F(0)=0,\quad F(1)=c,\quad F'(0)=a,\quad F'(1)=b.
\]
Moreover, $F$ is monotone increasing.
\item Let $\M^*(a,b,c)$ be as in \eqref{eq:Mstar}. We have
\begin{equation}\label{eq:regime1-quasi}
\eqindent
\norm{F''}_\Lii \le 6\M^*(a,b,c).
\end{equation}
\end{enumerate}
\end{theorem}

\begin{proof}

We start with Part (A).
Since \(\phi\in C^\infty([0,1])\) and the remaining factors in \eqref{eq:regime1-F} are affine,
we have \(F\in C^\infty([0,1])\subset C^2([0,1])\).
Using the endpoint values \(\phi(0)=1\) and \(\phi(1)=0\), we compute
\[
F(0)=(a\cdot 0)\phi(0) + (1-\phi(0))\bigl(b(-1)+c\bigr)=0,
\qquad
F(1)=(a\cdot 1)\phi(1) + (1-\phi(1))\bigl(b(0)+c\bigr)=c.
\]

Differentiating \eqref{eq:regime1-F}, we obtain
\begin{align*}
F'(x)
&= a\phi(x) + ax\phi'(x)
      -\phi'(x)\bigl(b(x-1)+c\bigr) + (1-\phi(x))\,b \\
&= b+(a-b)\phi(x)
   +\bigl(ax-(b(x-1)+c)\bigr)\phi'(x).
\end{align*}
Let
\begin{equation}\label{eq:D-def-refined}
D(x):=ax-\bigl(b(x-1)+c\bigr)=(a-b)x-(c-b).
\end{equation}
Then we have
\begin{equation}\label{eq:Fprime-formula-refined}
F'(x)=b+(a-b)\phi(x)+D(x)\phi'(x).
\end{equation}
Using \(\phi'(0)=\phi'(1)=0\) together with \(\phi(0)=1\) and \(\phi(1)=0\), we obtain
\[
F'(0)=b+(a-b)\cdot 1 + D(0)\cdot 0=a,
\qquad
F'(1)=b+(a-b)\cdot 0 + D(1)\cdot 0=b.
\]
Thus \(F\) matches the prescribed Hermite data.

We now verify monotonicity.
From \eqref{eq:D-def-refined} and the assumption \(c>\max\{a,b\}\), we have
\[
D(x)=(a-b)x-(c-b)\le \max\{0,a-b\}-(c-b)\le 0
\quad\text{for all }x\in[0,1].
\]
Inequality \eqref{eq:non-increasing} gives \(\phi'(x)\le 0\) on \([0,1]\).
Therefore, \(D(x)\phi'(x)\ge 0\) for all \(x\).
Substituting this into \eqref{eq:Fprime-formula-refined}, we obtain
\[
F'(x)\ge b+(a-b)\phi(x)
      =(1-\phi(x))\,b+\phi(x)\,a\ge 0,
\]
since \(0\le \phi\le 1\) and \(a,b\ge 0\).
Hence \(F\) is monotone, which proves (ii).

\smallskip
Now we turn to Part (B).
Differentiating \eqref{eq:Fprime-formula-refined} and using \(D'(x)=a-b\), we obtain $F''(x) = 2(a-b)\phi'(x) + D(x)\phi''(x)$. By the triangle inequality,
\begin{equation}
    \norm{F''}_\Lii
\le 2|a-b|\,\norm{\phi'}_\Lii
   +\norm{D}_\Lii\,\norm{\phi''}_\Lii.
   \label{eq:F bound whitney}
\end{equation}
Since \(D\) is affine, it attains its extrema at the endpoints $D(0) = -(c-b)$ and $D(1) = a-c$. Therefore, 
\begin{equation}
    \norm{D}_\Lii=c-\min\{a,b\} \leq \abs{2c - a - b}.
\label{eq:D bounds}
\end{equation}
Using \eqref{eq:phi-bounds} and \eqref{eq:D bounds} to estimate \eqref{eq:F bound whitney}, we see that \eqref{eq:regime1-quasi} follows.

\end{proof}

\subsection{Cubic \Bezier curve}
\label{sect:Bezier}

Throughout this subsection, we consider the region
\begin{equation}\label{eq:regime2-assump-graph}
\min\set{a,b} \leq c \leq  \max\set{a,b}.
\end{equation}
We construct an interpolant as a parametric cubic \Bezier curve in the \((x,y)\)-plane,
with a repeated interior control point chosen at the intersection of the endpoint tangent lines.
We then verify that the resulting curve is the graph of a \(C^2\) function, satisfies the Hermite data, and is monotone.

Define
\[
\ell_0(x):=ax,\qquad \ell_1(x):=b(x-1)+c.
\]
Thanks to \eqref{eq:regime2-assump-graph}, $\ell_1(0)-\ell_0(0)=c-b$ and $\ell_1(1)-\ell_0(1)=c-a$ have different signs, so there exists a unique \(T\in(0,1)\) such that \(\ell_0(T)=\ell_1(T)\).
Solving \(aT=b(T-1)+c\), we obtain
\begin{equation}\label{eq:T-def}
T=\frac{b-c}{b-a},\qquad
m:=aT=\frac{a(b-c)}{b-a}.
\end{equation}
We define \(\Gamma:[0,1]\to\mathbb R^2\) by
\begin{equation}\label{eq:param-bezier}
\Gamma(t)=(x(t),y(t))
:=3(1-t)^2t\,(T,m)+3(1-t)t^2\,(T,m)+t^3(1,c).
\end{equation}
After simplification, this can be written as
\begin{equation}\label{eq:x-y-explicit}
x(t)=3T\,t(1-t)+t^3,\qquad
y(t)=3m\,t(1-t)+c\,t^3.
\end{equation}
We note that \(\Gamma(0)=(0,0)\) and \(\Gamma(1)=(1,c)\).

The \Bezier construction need not be quasi-optimal throughout the full range \(a<c<b\). In fact, as \(c\) approaches either endpoint \(a\) or \(b\), the curvature of the resulting \Bezier interpolant blows up. More precisely, we have the following result.

\begin{theorem}\label{thm:bezier cubic}
Under \eqref{eq:regime2-assump-graph}, the curve \(\Gamma\) in \eqref{eq:param-bezier} defines a function
\(G\in C^\infty([0,1])\subset C^2([0,1])\) via \(G(x(t))=y(t)\). Moreover:
\begin{enumerate}[(A)]
\item \(G(0)=0\), \(G(1)=c\), \(G'(0)=a\), and \(G'(1)=b\).
\item \(G'(x)\ge 0\) for all \(x\in[0,1]\).
\item For \(t\in[0,1]\),
\begin{equation}\label{eq:Gpp-param-formula}
\eqindent
G''(x(t))
=
\frac{y''(t)x'(t)-y'(t)x''(t)}{(x'(t))^3}
=
\frac{18\,t(1-t)\,(Tc-m)}{(x'(t))^3}.
\end{equation}
In particular,
\begin{equation}\label{eq:curv-at-T-correct}
\eqindent
G''(x(T))
=
\frac{2}{3}\,\frac{(b-a)^3}{(b-c)(c-a)},
\end{equation}
and as a consequence,
\begin{equation*}
    \eqindent
\norm{G''}_\Lii
\ge
\frac{2}{3}\,\frac{(b-a)^3}{(b-c)(c-a)}.
\end{equation*}
\end{enumerate}
\end{theorem}

\begin{proof}
Without loss of generality, we may assume that $a < b$. The case $a > b$ can be proved similarly with the roles flipped.

Differentiating \eqref{eq:x-y-explicit}, we obtain
\[
x'(t)=3T(1-2t)+3t^2
=3\bigl((t-T)^2+T(1-T)\bigr).
\]
Since \(T\in(0,1)\), we have \(x'(t)>0\) for all \(t\in[0,1]\).
Thus \(x\) is strictly increasing and maps \([0,1]\) onto \([0,1]\).
Since \(x\in C^\infty\) and \(x'(t)\neq 0\), the inverse function theorem yields that \(x^{-1}\in C^\infty\).
We define \(G(x):=y(x^{-1}(x))\), which gives \(G\in C^\infty([0,1])\subset C^2([0,1])\).

Differentiating \(y\) and \(x\), we obtain
\[
y'(t)=3m(1-2t)+3ct^2,\qquad x'(t)=3T(1-2t)+3t^2.
\]
A direction computation using $G'(x(t))=\frac{y'(t)}{x'(t)}$, $m = aT$, $T = \frac{b-c}{b-a}$ gives 
\[
G'(0)=\frac{3m}{3T}=\frac{m}{T}=a
\text{\quad and \quad}
G'(1)=\frac{c-m}{1-T} = b.
\]

To verify monotonicity, we note that \(x'(t)>0\), so it suffices to show \(y'(t)\ge 0\).
We write
\[
y'(t)=3(ct^2-2mt+m)
=3\Bigl[c\Bigl(t-\frac{m}{c}\Bigr)^2+m\Bigl(1-\frac{m}{c}\Bigr)\Bigr].
\]
Since \(m=aT<a<c\), we have \(m/c<1\), hence the expression is strictly positive for all \(t\).
Thus \(G'(x)\ge 0\).

Differentiating again, we obtain $x''(t)=6(t-T)$ and $y''(t)=6(ct-m)$.
Substituting these into the parametric formula for the second derivative, we have
\[
G''(x(t))=\frac{y''(t)x'(t)-y'(t)x''(t)}{(x'(t))^3},
\]
a direct calculation yields
\[
y''(t)x'(t)-y'(t)x''(t)=18\,t(1-t)\,(Tc-m),
\]
which proves \eqref{eq:Gpp-param-formula}. Next, plugging $x'(T)=3T(1-T)$ into \eqref{eq:Gpp-param-formula}, we see that \eqref{eq:curv-at-T-correct} holds. 

\end{proof}

\subsection{Quartic interpolants via Bernstein basis}
\label{sect:Bernstein}

Next, we give ourselves one extra degree of freedom and use Bernstein polynomials to construct quartic solutions. We relax our previous assumption and consider general $a,b,c \geq 0$. 

Recall that the family of Bernstein basis polynomials of degree $N$ $\set{B_{k,N}(t): k = 0,1,\cdots,N}$ on \([0,1]\) are given by $B_{k,N}(t)=\binom{N}{k} t^k(1-t)^{N-k}$ with the properties that
\begin{equation}
B_{k,N}(t)\ge 0 \text{\quad and \quad} \sum_{k=0}^N B_{k,N}(t)=1 \quad \text{ for all } k \in \set{0,1,\cdots,N} \text{ and } t\in[0,1].
\label{eq:Berstein properties}
\end{equation}

\begin{theorem}
\label{thm:bernstein generalized}
Let $m_0 = a$ and $m_2 = b$. Suppose that $a,b,c \in [0,\infty)$ satisfy
\begin{equation}
    \eqindent
    \abs{c-\frac{a+b}{2}} \leq \frac{\lambda}{4}\abs{b-a}.
    \label{eq:Bezier range restriction}
\end{equation}
for some $\lambda$ satisfying
    \begin{equation}
        \eqindent
        0 < \lambda \leq \frac{a+b}{\abs{b-a}}.
        \label{eq:lambda assumption}
    \end{equation}
Then there exist $m_1, m_2 \geq 0$ with
    \begin{equation}
    \eqindent
        \abs{m_1 - a}\leq \lambda\abs{b-a} \text{\quad and \quad} \abs{m_2 - b}\leq \lambda\abs{b-a},
        \label{eq:mk assumption}
    \end{equation}
    such that the Bernstein interpolant given by 
    \begin{equation*}
    \eqindent
    G(t) := \int_{0}^t \sum_{k=0}^3m_kB_{k,3}(\tau)d\tau
\end{equation*}
    is monotone, 
     satisfies $G(0) = 0$, $G(1) = c$, $G'(0) = a$, $G'(1) = b$, and
\begin{equation}\label{eq:quasioptimal-bound-corrected}
\eqindent
\norm{G''}_\Lii\le  3(2\lambda + 1)\M^*(a,b,c).
\end{equation}

\end{theorem}

\begin{remark}
    The control of the curvature in \eqref{eq:quasioptimal-bound-corrected} is \emph{not} universal and depends on the distance between $c$ and $\frac{a+b}{2}$ relative to $\abs{b-a}$. One can be safe and pick $\lambda = 1$ to handle the case when $c$ is ``well situated near $\frac{a+b}{2}$''. 
\end{remark}

\begin{proof}[Proof of Theorem \ref{thm:bernstein generalized}]
For convenience, set
\begin{equation*}
    v(t)=aB_{0,3}(t)+m_1B_{1,3}(t)+m_2B_{2,3}(t)+bB_{3,3}(t)
    \text{\quad so \quad}
    G(t):= \int_{0}^t v(\tau)d\tau.
\end{equation*}

The admissibility condition \(G(1) = \int_0^1 v=c\), together the fact that $\int_0^1 B_{k,3} = \tfrac{1}{4}$, is equivalent to the solvability of the following equation:
\begin{equation}
    m_1+m_2=4c-a-b.
    \label{eq:mk and ab}
\end{equation}
Thanks to \eqref{eq:lambda assumption} and \eqref{eq:Bezier range restriction}, we see that
\begin{equation*}
    c \geq \frac{a+b}{2} - \frac{\lambda}{4}\abs{b-a} \geq \frac{a+b}{4}
    \quad\Longleftrightarrow\quad
    4c-a-b \geq 0.
\end{equation*}
Thanks to \eqref{eq:Bezier range restriction} again, we see that
\begin{equation*}
    \abs{4c - 2(a+b)} \leq \lambda \abs{b-a}
    \quad \Longleftrightarrow\quad
    (a+b)-\lambda\abs{b-a}\leq 4c - a - b \leq (a+b)+\lambda\abs{b-a}.
\end{equation*}
Therefore, we can solve \eqref{eq:mk and ab} with $m_1, m_2 \geq 0$ satisfying \eqref{eq:mk assumption}.

In view of \eqref{eq:lambda assumption} and \eqref{eq:mk assumption}, we have
\begin{equation*}
    m_1, m_2 \geq 0.
\end{equation*}
Recall that $m_0 = a$ and $m_3 = b$. Thanks to \eqref{eq:Berstein properties},
\begin{equation*}
    G'(t) = v(t) \geq 0 \text{\quad for all $t \in [0,1]$.}
\end{equation*}
Namely, $G$ is monotone increasing. 

Now we esimate the curvature. 
\begin{equation}
    v'(t)=3\sum_{k=0}^2(m_{k+1}-m_k)B_{k,2}(t). 
    \label{eq:vprime in bernstein}
\end{equation}
Thanks to \eqref{eq:mk assumption}, we have
\begin{equation}
    \abs{m_2 - m_1} = \abs{(m_2 - b) + (b-a) + (a-m_1)}\leq (2\lambda + 1)\abs{b-a}.
    \label{eq:mk assumption - 2}
\end{equation}
Using \eqref{eq:Berstein properties}, \eqref{eq:mk assumption}, and \eqref{eq:mk assumption - 2} to estimate \eqref{eq:vprime in bernstein}, we have
\begin{equation*}
    \norm{G''}_\Lii = \norm{v'}_\Lii \leq 3(2\lambda + 1)\abs{b-a} \leq 3(2\lambda+1)\Mloc(a,b,c).
\end{equation*}
\end{proof}

\section{Velocity-based constructions for general $a,b,c \geq 0$}
\label{section:velocity based}

 Recall from Section \ref{section: classical techniques} that the Whitney extension operator only applies in the range $c \geq \max\set{a,b}$ and breaks monotonicity when this condition is violated, and 
 the \Bezier construction (Theorems \ref{thm:bezier cubic} and \ref{thm:bernstein generalized}) only applies when $c$ is relatively close to $\frac{a+b}{2}$. In this section, we provide an alternative treatment for general $a,b,c \in [0,\infty)$ with $c > 0$. In particular, this covers the case when the \Bezier and Bernstein construction (Sections \ref{sect:Bezier} and \ref{sect:Bernstein}) is no longer stable or quasi-optimal.

The starting point is the equivalent formulation in terms of the velocity \(v = G'\) similar to that of Theorem \ref{thm:bernstein generalized}, where one seeks a nonnegative function \(v\) with prescribed endpoint values and mass constraint,
\[
v(0)=a,\qquad v(1)=b,\qquad v \ge 0,\qquad \int_0^1 v = c,
\]
while minimizing $\norm{v'}_\Lii$. In the current setting, the optimal \(C^{1,1}\) solution develops a one-sided saturation: the minimizing profile is no longer symmetric, but instead depends on the position of \(c\) relative to 
\begin{equation}
    c_0 := \frac{a^2 + b^2}{2(a+b)} < \frac{a+b}{2}
    \label{eq:c_0}
\end{equation}
for $a,b$ not both zero.

We will begin our analysis by first proving Theorem \ref{thm:main - Mstar}, which will determine the minimal curvature $\Mloc(a,b,c)$ in \eqref{eq:Mstar inf def}. We will then prove Theorem \ref{thm:explicit complete}, which an explicit monotone $C^{1,1}$ interpolant realizing $\Mloc(a,b,c)$.

Lastly, in Theorem \ref{thm:C2 mollify}, we will construct an explicit \(C^2\) interpolant $\widetilde{G}$ by smoothing the jump discontinuities in the second derivative. The mollification is performed locally near the junction points and preserves both the endpoint constraints and monotonicity. Moreover, the curvature of the smoothed interpolant remains close to the optimal value
\[
\norm{\widetilde{G}''}_\Lii\le (1+c_0)\, \M^*(a,b,c)
\]
with \(c_0=0.2\) in our construction. We do not claim that this constant is the best possible. 

\medskip
This provides a complete solution to the optimal $C^{1,1}$ (and quasi-optimal $C^2$) monotone Hermite interpolation problem. 


\subsection{Proof of Theorems \ref{thm:main - Mstar} and \ref{thm:explicit complete}}

\begin{proof}[Proof of Theorem \ref{thm:main - Mstar}]

Suppose $a = b = c = 0$. Then the constant function $G \equiv 0$ is admissible for $\Mloc(a,b,c)$, so $\Mloc(a,b,c) = 0$. Suppose $a+b > 0$ but $c = 0$, then the Hermite data and monotonicity are inconsistent in the class of $C^1$, so $\Mloc(a,b,c) = \infty$. For the rest of the proof, we assume that $c > 0$ whenever $a + b > 0$. 

Write \(v:=G'\). Then admissibility of \(G\) is equivalent to
\[
v\in C^1([0,1]),\quad v\ge 0,\quad v(0)=a,\quad v(1)=b,\quad \int_0^1 v=c,
\]
and
\[
\norm{G''}_\Lii = \norm{v'}_\Lii.
\]
Thus \(\M^*(a,b,c)\) is the infimum of $\norm{v'}_\Lii$ over all such \(v\).

Let \(v\) be admissible, and we set \(M:=\norm{v'}_\Lii\). This means that
\begin{equation}
v(0) = a\text{\quad,\quad}
v(1) = b\text{\quad,\quad}
\int_{0}^1 v = c \text{\quad, and \quad}
\norm{v'}_\Lii \leq M.
\label{eq:main v admissible}
\end{equation}

We consider the following two cases separately.
\begin{enumerate}[\text{Case } I.]
\item $c \leq \frac{a+b}{2}$.
\item $c > \frac{a+b}{2}$.
\end{enumerate}

\paragraph{Proof of Case I: when $c \leq \frac{a+b}{2}$.}

Define
\begin{equation*}
\phi_M^-(t) := \max\set{0,\ a-Mt,\ b-M(1-t)}.
\end{equation*}
In view of \eqref{eq:main v admissible}, we have
\begin{equation}
v(t) \geq \phi_M^-(t) \text{\quad for all $t \in [0,1]$.}
\label{eq:vt complete}
\end{equation}
Integrating and using \eqref{eq:vt complete}, we have
\[
\Phi^-(M):= \int_{0}^1 \phi_M^- \geq \int_{0}^{1}v = c.
\]
We now compute \(\Phi^-(M)\) by distinguishing the two cases.

The affine functions $a-Mt$ and $b-M(1-t)$ intersect at $t_*=\frac{1}{2} + \frac{a-b}{2M}$ with value $\frac{a+b}{2} - \frac{M}{2}$.
\begin{itemize}
    \item If \(\abs{b-a}\le M\le a+b\), then \(\phi_M^-\) is given by the two linear pieces, and a direct calculation yields
\begin{equation}
    \Phi^-(M)=\frac{a+b}{2}-\frac{M}{4}+\frac{(b-a)^2}{4M}.
    \eqindent
    \label{eq:PhiM when M small}
\end{equation}
\item If \(M\ge a+b\), then $\phi_M^-$ is given by three linear pieces with corners at $t_0 = \frac{a}{M}$ and $t_1 = 1-\frac{b}{M}$, and one obtains
\begin{equation*}
    \Phi^-(M) = \frac{a^2 + b^2}{2M}.
\end{equation*}
\end{itemize}
In both cases, \(\Phi^-\) is continuous and strictly decreasing on \([\abs{b-a},\infty)\). Therefore any admissible \(v\) must satisfy \(c\ge \Phi^-(M)\), which implies
\[
M\ge (\Phi^-)^{-1}(c),
\]
and hence
\begin{equation}
    \M^*(a,b,c)\ge (\Phi^-)^{-1}(c).
    \label{eq:Mstar lower bound}
\end{equation}
To prove the reverse inequality, define
\[
M_c:=(\Phi^-)^{-1}(c),
\qquad
v_c(t):=\phi_{M_c}^-(t).
\]
Then \(v_c\) is piecewise linear and satisfies all constraints. Moreover, 
\[
\norm{v_c'}_\Lii\le M_c.
\]
Integrating $v_c$ yields $G_c\in C^{1,1}([0,1])$ with
\begin{equation}
    \M^*(a,b,c)\le\norm{G_c''}_\Lii\le M_c.
\label{eq:Mstar upper bound}
\end{equation}
Combining both bounds \eqref{eq:Mstar lower bound} and \eqref{eq:Mstar upper bound}, we have
\[
\M^*(a,b,c)=(\Phi^-)^{-1}(c).
\]
We now explicitly solve for $M_c$.
\begin{itemize}
    \item If \(M_c\ge a+b\), then \(c=\frac{a^2+b^2}{2M_c}\), giving
\[
M_c=\frac{a^2+b^2}{2c},
\]
which corresponds to \(c\le c_0\).

\item If \(M_c\le a+b\), we can solve the quadratic equation $\Phi(M_c) = c$ and obtain
\[
M_c=(a+b-2c)+\sqrt{(a+b-2c)^2+d^2},
\]
which corresponds to $c_0 < c \leq \frac{a+b}{2}$
\end{itemize}

\paragraph{Proof of Case II: when $c > \frac{a+b}{2}$.} Define
\begin{equation}
\phi_M^+(t) := \min\set{a+Mt, b+M(1-t)}.
\end{equation}
In view of \eqref{eq:main v admissible}, we have 
\begin{equation}
v(t) \leq \phi_M^+ \text{\quad for all $t\in[0,1]$.}
\label{eq:vt complete - 2}
\end{equation}
In this case, $\phi_M^+$ consists of only two affine pieces, and they intersect at $t_* = \frac{1}{2} + \frac{b-a}{2M}$ with value $\frac{a+b}{2} + \frac{M}{2}$. Integrating, we have
\begin{equation*}
    \Phi^+(M) := \int_{0}^1\phi_M^+ = \frac{a+b}{2} + \frac{M}{4} - \frac{(b-a)^2}{4M} \geq \int_0^1v = c.
\end{equation*}
One verifies that $\Phi^+$ is increasing in $M$ for $M \geq \abs{b-a}$. By a symmetric argument as above, we can solve for $\Phi^+(M_c) = c$ and obtain
\begin{equation*}
    M_c = (2c-a-b) + \sqrt{(2c-a-b)^2 + (b-a)^2},
\end{equation*}
which occurs when $c \geq \frac{a+b}{2}$.

This completes the proof of Theorem \ref{thm:main - Mstar}.
\end{proof}

    \begin{proof}[Proof of Theorem \ref{thm:explicit complete}]
        Set $v(t):= \phi_M^\pm(t)$ as in \eqref{eq:vt complete} or \eqref{eq:vt complete - 2} depending on the value of $c$. A direct integration in the three separate cases yield the formulae. The rest of the conclusions follow from Theorem \ref{thm:main - Mstar}.
    \end{proof}

\begin{remark}
    \newcommand{\conv}{\mathsf{conv}}
    There is an alternative proof of Theorem \ref{thm:explicit complete} when $c \geq c_0$ which relies on much deeper results from \cite{legruyer2009minimal,azagra2017}. For a real-valued function $g: \R \to \R$, we use $\conv(g)$ to denote the (lower) convex envelope of $g$ defined by
\begin{equation*}
    \conv(g)(t) := \sup\set{h(t): \text{$h$ is convex, proper, lower semicontinuous, and $h\leq g$}}.
\end{equation*}
Rewrite the Hermite data as $\ell_0(t) = at$ and $\ell_1(t) = b(t-1) + c$. It was shown in \cite{azagra2017} that the formula 
\begin{equation}
    F = - \frac{M}{2}t^2 + \conv\brac{
    \inf_{k = 0,1}\set{
    \ell_k(k) + \ell_k'(t - k) + \frac{M}{2}\abs{t-k}^2
    } + \frac{M}{2}t^2
    }
    \label{eq:azagra formula}
\end{equation}
gives a $C^{1,1}$ interpolant matching the Hermite data with $\Lip(G') \leq M$ whenever $M \geq \M^*(a,b,c)$.
Moreover, it was shown in \cite{legruyer2009minimal} that $\Mloc(a,b,c)$ (when $c \geq c_0$) is the \emph{best possible} Lipschitz bound for the derivatives of any $C^{1,1}$ interpolants matching the Hermite data, although monotonicity was not a consideration. After an elementary but tedious calculation to find a tangent line to two parabolas, we can verify that the exact same formula \eqref{eq:azagra formula} give rise to the same $G(t)$ in Theorem \ref{thm:explicit complete} when $c \geq c_0$. However, the formula does not preserve monotonicity when $c < c_0$.
\end{remark}

\subsection{Near-optimal $C^2$-smoothing}
\label{sect:C2 smoothing}

In this subsection, we smooth the $C^{1,1}$ solution in Theorem \ref{thm:explicit complete} into a $C^2$ solution sacrificing a small amount of optimality. 

\begin{theorem}[Near-optimal $C^2$-smoothing]\label{thm:C2 mollify}
    Let $a,b,c \in [0,\infty)$ with $c > 0$. There exists an explicit admissible interpolant \(G\in C^2([0,1])\) such that
\[
G(0)=0,\qquad G(1)=c,\qquad G'(0)=a,\qquad G'(1)=b,\qquad G'\ge 0,
\]
and
\[
\norm{G''}_\Lii\le (1+c_0)\M^*(a,b,c).
\]
Moreover, we may take $c_0 = 0.2$. 

\end{theorem}

\begin{proof}
    Let $M := \Mloc(a,b,c)$. 
    Let 
    \begin{equation*}
        v_M(t) := \phi_M^\pm(t)
    \end{equation*}
    with $\phi_M^\pm$ as in \eqref{eq:vt complete} or \eqref{eq:vt complete - 2}, depending on the value of $c$. We will prove the case $v_M(t) = \phi_M^-(t)$, as the proof of the other case is almost identical. In this case,
    \begin{equation*}
    v_M(t) = \max\set{0, a-Mt, b-M(1-t)}.
    \end{equation*}
     We will mollify \(v_M\) and integrate the result to produce an admissible \(C^2\) interpolant with curvature bounded by \((1+c_0)M\), 
    
    Note that $v_M$ is piecewise affine with slopes in \(\{-M,0,M\}\) and its graph has at most two corners. Let \(\Sigma\subset(0,1)\) denote the $x$-coordinates of the corners. 
Choose \(\delta > 0\) such that
\begin{equation}\label{eq:delta-choice-unified-final}
\delta\le \min\Biggl\{
\frac12 \inf_{\tau\in\Sigma}\min\{\tau,1-\tau\},
\frac14 \min_{\substack{\tau,\tau'\in\Sigma\\ \tau\neq\tau'}}|\tau-\tau'|,
\frac{b}{16M}
\Biggr\},
\end{equation}
with the convention that empty minima are omitted.
Then the intervals
\[
[\tau-\delta,\tau+\delta],\qquad \tau\in\Sigma,
\]
are pairwise disjoint, lie inside \((0,1)\), and are disjoint from
\[
J_\delta:=[1-6\delta,\,1-2\delta].
\]

Fix \(\tau\in\Sigma\) and let \(s_-\) and \(s_+\) be the left and right slopes of \(v_M\) at \(\tau\), i.e., $s_\pm = v_M'(\tau \pm)$.
Since \(v_M\) is a lower-envelope profile, its slope is nondecreasing across corners, so
\[
(s_-,s_+)\in\{(-M,0),\ (0,M),\ (-M,M)\}.
\]
In particular, \(s_-\le s_+\).
Define on \([\tau-\delta,\tau+\delta]\)
\begin{equation}\label{eq:quadratic-splice-final}
q_\tau(t):=
v_M(\tau-\delta)
+s_-\,(t-(\tau-\delta))
+\frac{s_+-s_-}{4\delta}(t-(\tau-\delta))^2.
\end{equation}
Differentiating \eqref{eq:quadratic-splice-final}, we obtain
\begin{equation*}
    q_\tau'(t)=s_-+\frac{s_+-s_-}{2\delta}(t-(\tau-\delta)).
\end{equation*}
Evaluating at \(t=\tau-\delta\) and \(t=\tau+\delta\), we see that
\begin{equation*}
    q_\tau'(\tau-\delta)=s_- \text{\quad and \quad} q_\tau'(\tau+\delta)=s_-+\frac{s_+-s_-}{2\delta}(2\delta)=s_+.
\end{equation*}
Similarly,
\begin{equation*}
  q_\tau(\tau-\delta)=v_M(\tau-\delta)  \text{\quad and \quad} q_\tau(\tau+\delta) = v_M(\tau-\delta)+\delta(s_-+s_+) = v_M(\tau+\delta).
\end{equation*}
The last equality holds because the left and right affine branches of \(v_M\) meet at \(\tau\), and the value change from \(\tau-\delta\) to \(\tau+\delta\) is exactly \(\delta(s_-+s_+)\).
therefore, \(q_\tau\) matches both value and first derivative of the adjacent affine pieces at the endpoints.

Define
\[
w_\delta(t):=
\begin{cases}
q_\tau(t), & t\in[\tau-\delta,\tau+\delta]\text{ for some }\tau\in\Sigma,\\[0.4em]
v_M(t), & \text{otherwise}.
\end{cases}
\]
Then \(w_\delta\in C^1([0,1])\), and
\[
w_\delta(0)=a,\qquad w_\delta(1)=b.
\]

Next, we verify that
\begin{equation}\label{eq:wdelta-properties-final}
w_\delta\ge v_M\ge 0
\text{\quad and \quad}
\norm{w_\delta'}_\Lii\le M.
\end{equation}
On each smoothing interval, \(q_\tau'\) varies affinely between \(s_-\) and \(s_+\), and both satisfy
\[
|s_-|\le M,\qquad |s_+|\le M.
\]
Hence
\[
|q_\tau'(t)|\le M
\qquad\text{for all }t\in[\tau-\delta,\tau+\delta].
\]
Outside the smoothing intervals, \(w_\delta'=v_M'\in\{-M,0,M\}\), so
\[
\norm{w_\delta'}_\Lii\le M.
\]

To prove \(w_\delta\ge v_M\), first note that on the left half of a smoothing interval, \(v_M\) agrees with the affine function
\[
L_\tau(t):=v_M(\tau-\delta)+s_-(t-(\tau-\delta)).
\]
Subtracting, we obtain
\[
q_\tau(t)-L_\tau(t)
=
\frac{s_+-s_-}{4\delta}(t-(\tau-\delta))^2\ge 0.
\]
Thus \(q_\tau\ge v_M\) on the left half.
On the right half, \(v_M\) agrees with the affine function tangent to \(q_\tau\) at \(t=\tau+\delta\).
Since \(q_\tau\) is convex, it lies above its tangent line there. Hence \(q_\tau\ge v_M\) on the right half as well.
Therefore \(w_\delta\ge v_M\ge 0\) on all of \([0,1]\).

Since we have modified the velocity, we have to correct the excess area of error. Set
\[
E_\delta:=\int_0^1 w_\delta(t)\,dt-c
=\int_0^1 (w_\delta(t)-v_M(t))\,dt.
\]
Then \(E_\delta\ge 0\).
We now estimate this excess area.
\begin{itemize}
    \item If \((s_-,s_+)=(-M,0)\), then on the smoothing interval,
\[
q_\tau(t)-v_M(t)=\frac{M}{4\delta}(t-(\tau-\delta))^2
\]
on the left half, and an analogous quadratic expression on the right half.
Integrating over the entire smoothing interval gives $E_\delta = \frac{M\delta^2}{6}$. 
The same value is obtained when \((s_-,s_+)=(0,M)\).

\item If \((s_-,s_+)=(-M,M)\), then the coefficient doubles, and the corresponding area gain $E_\delta = \frac{M\delta^2}{3}$.
\end{itemize}

Since \(v_M\) has at most two corners, and the only possibilities are either two one-sided corners or one two-sided corner, we obtain the uniform estimate
\begin{equation}\label{eq:Edelta-bound-final}
0\le E_\delta\le \frac13 M\delta^2.
\end{equation}

We now restore the correct area by subtracting a smooth bump.
Define
\[
B(s):=
\begin{cases}
30s^2(1-s)^2, & 0\le s\le 1,\\
0, & \text{otherwise}.
\end{cases}
\]
Then \(B\in C^1(\mathbb R)\), \(B\ge 0\), and $\int_{0}^1B = 1$.
Define
\[
\eta_\delta(t):=\frac{1}{4\delta}\,
B\!\left(\frac{t-(1-6\delta)}{4\delta}\right).
\]
Then \(\eta_\delta\in C^1([0,1])\), \(\eta_\delta\ge 0\),
\[
\operatorname{supp}(\eta_\delta)\subset J_\delta,
\qquad
\int_0^1 \eta_\delta(t)\,dt=1.
\]
Finally, we set
\[
v_\delta(t):=w_\delta(t)-E_\delta\,\eta_\delta(t).
\]
Since \(\eta_\delta\) is supported away from the endpoints, $v_\delta(0)=a$, $v_\delta(1)=b$, and
\begin{equation}
    \int_0^1 v_\delta(t)\,dt
=
\int_0^1 w_\delta(t)\,dt-E_\delta\int_0^1 \eta_\delta(t)\,dt
=
(c+E_\delta)-E_\delta
=
c.
\end{equation}

It remains to verify $v_\delta \geq 0$.
On \(J_\delta=[1-6\delta,1-2\delta]\), the function \(v_M\) lies on its final affine branch, so
\[
v_M(t)=b-M(1-t)\ge b-6M\delta.
\]
Using \(\delta\le b/(16M)\), we obtain
\[
v_M(t)\ge b-\frac{6}{16}b=\frac58 b.
\]
Since \(w_\delta\ge v_M\), we have
\[
w_\delta(t)\ge \frac58 b
\qquad\text{on }J_\delta.
\]

Next, since
\[
\sup_{s\in\mathbb R} B(s)=\sup_{0\le s\le 1} 30s^2(1-s)^2
=
30\left(\frac14\right)^2
=
\frac{30}{16},
\]
we obtain
\[
\norm{\eta_\delta}_\Lii\le \frac{30}{64\delta}.
\]
Using \eqref{eq:Edelta-bound-final}, we have
\[
E_\delta \cdot \sup \eta_\delta
\le \frac13 M\delta^2\cdot \frac{30}{64\delta}
=
\frac{5}{32}M\delta
\le \frac{5}{512}b
<\frac18 b.
\]
Therefore, on \(J_\delta\), $v_\delta(t)\ge \frac58 b-\frac18 b=\frac12 b>0$. Outside \(J_\delta\), we have \(\eta_\delta=0\), so \(v_\delta=w_\delta\ge 0\).
Thus
\[
v_\delta\ge 0
\qquad\text{on }[0,1].
\]

Finally, we estimate the derivative. Since $v_\delta'=w_\delta'-E_\delta\eta_\delta'$, we obtain
\[
\norm{v_\delta'}_\Lii
\le \norm{w_\delta'}_\Lii + E_\delta \norm{\eta_\delta'}_\Lii.
\]
Differentiating \(\eta_\delta\), we obtain
\[
\eta_\delta'(t)
=
\frac{1}{16\delta^2}
B'\!\left(\frac{t-(1-6\delta)}{4\delta}\right).
\]
Thus
\[
\norm{\eta_\delta'}_\Lii
\le \frac{1}{16\delta^2}\norm{B'}_\Lii.
\]
Since $B'(s)=60s(1-s)(1-2s)$, we can verify that $\norm{B'}_\Lii=\frac{10}{\sqrt3}$, so
\[
\norm{\eta_\delta'}_\Lii<\frac{10}{16\sqrt{3}\delta^2}.
\]
Combining this with \eqref{eq:wdelta-properties-final} and \eqref{eq:Edelta-bound-final}, we obtain
\begin{equation}
\norm{v_\delta'}_\Lii
\le M+\frac13 M\delta^2\cdot \frac{10}{16\sqrt{3}\delta^2}
\leq 1.2 M.
    \label{eq:v_delta estimate}
\end{equation}
Define
\[
G(x):=\int_0^x v_\delta(t)\,dt.
\]
Since \(v_\delta\in C^1([0,1])\), we have \(G\in C^2([0,1])\). Moreover, $G(0) = 0$, $G(1) = \int_{0}^1 v_\delta = c$, $G'(0)=v_\delta(0)=a$, $G'(1)=v_\delta(1)=b$, and 
\[
G'(x)=v_\delta(x)\ge 0
\qquad\text{for all }x\in[0,1].
\]
Finally, thanks to \eqref{eq:v_delta estimate},
\[
\norm{G''}_\Lii
=
\norm{v_\delta'}_\Lii
<1.2M
=
1.2\Mloc(a,b,c).
\]
This completes the proof.
\end{proof}

\section{Numerical examples}

We implemented all four interpolation methods in this paper for the two-point setting, and we implemented the velocity-based method to handle $C^{1,1}$ monotone Hermite interpolation problem of arbitrary (finite) size. The code used to produce these figures is available on our GitHub repository
\begin{center}
    \url{https://github.com/jfsmath/Monotone-Interpolation}.
\end{center}

In Figures \ref{fig:average c} and \ref{fig:average c large b-a}, we compare all four methods in this paper with data $c \approx \frac{a+b}{2}$. In Figure \ref{fig:average c}, we also have $a \approx b$, and monotonicity is preserve across all four methods. In Figure \ref{fig:average c large b-a}, we have $a \gg b$, and the classical Whitney extension operator (Theorem \ref{thm:whitney extension operator}) fails to preserve monotonicity.

\begin{figure}[htbp]
    \centering
    \begin{subfigure}{\linewidth}
        \includegraphics[width=\linewidth]{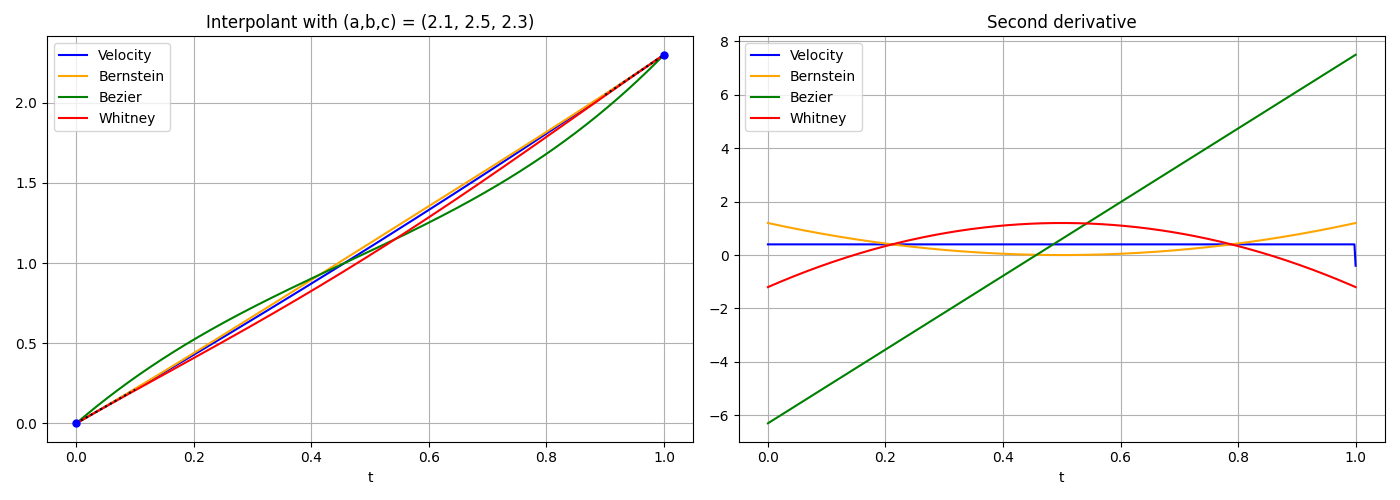}
    \caption{$a \approx b \approx c$}
    \label{fig:average c}
    \end{subfigure}
    \\[1em]
    \begin{subfigure}{\linewidth}
        \includegraphics[width=\linewidth]{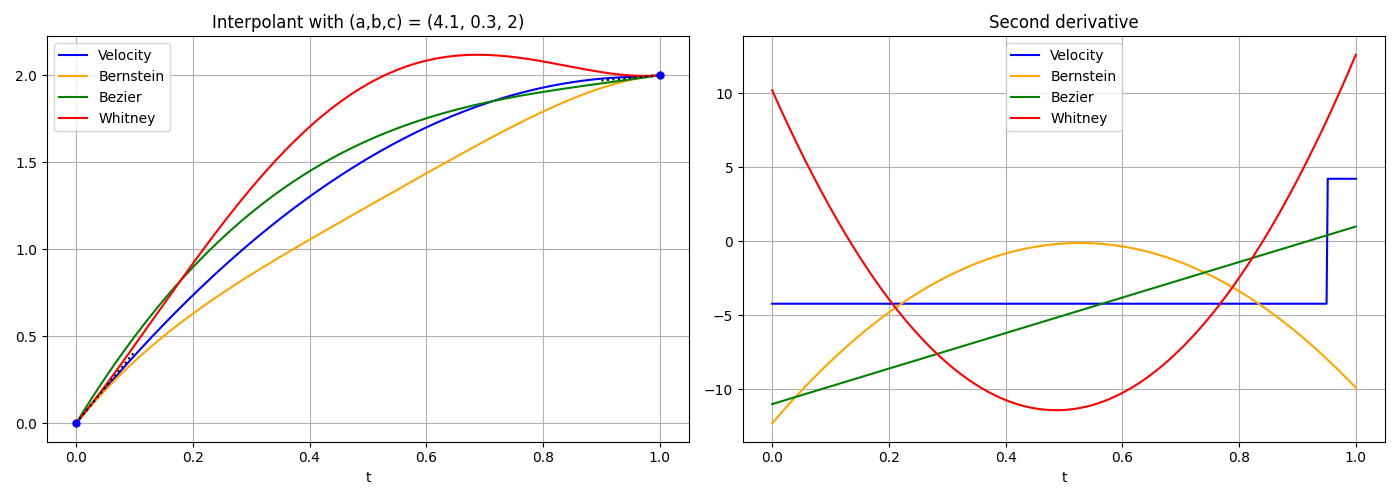}
    \caption{$a \gg b$}
    \label{fig:average c large b-a}
    \end{subfigure}
    \caption{Method comparison when $c \approx \frac{a+b}{2}$}
\end{figure}

In Figures \ref{fig:large c no bezier} and \ref{fig:large c no bezier}, we consider the case $c \gg \frac{a+b}{2}$. In this case, both the Whitney and the Bernstein methods preserve monotonicity. The \Bezier curve is ill-behaved so it is not included in the plot.  

\begin{figure}[htbp]
    \centering
    \includegraphics[width=\linewidth]{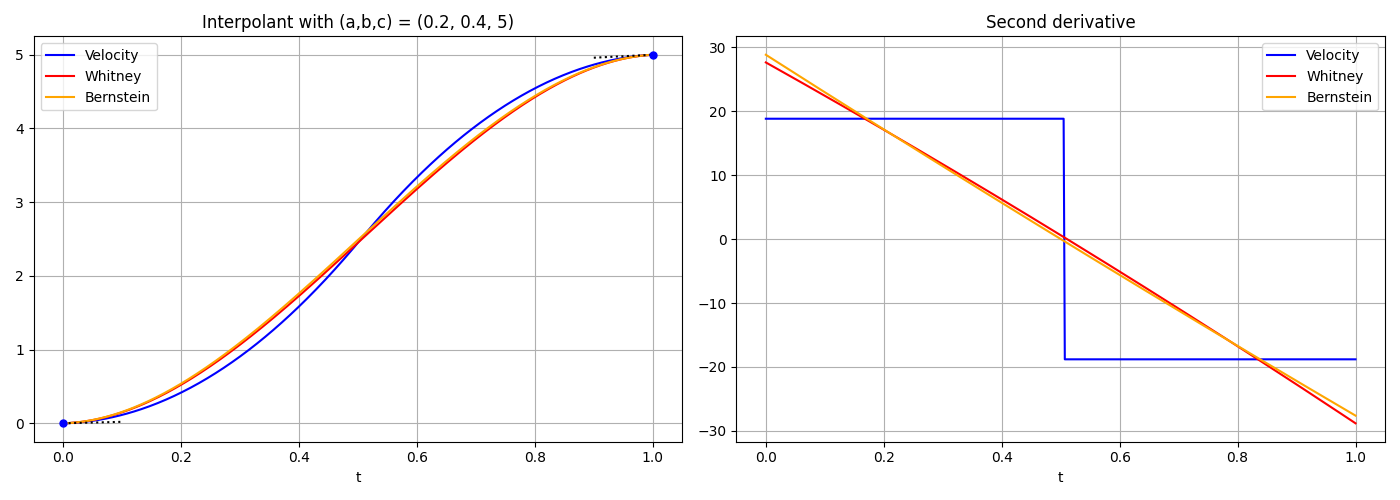}
    \caption{Method comparison when $c \gg \frac{a+b}{2}$}
    \label{fig:large c no bezier}
    \end{figure}

In Figure \ref{fig:small c, no bezier}, we consider the case $c \ll \frac{a+b}{2}$. In this case, neither the Whitney nor the Bernstein method preserves monotonicity. The \Bezier method is again ill-behaved so we do not include it in the plot. 

\begin{figure}[htbp]
    \centering
    \includegraphics[width=\linewidth]{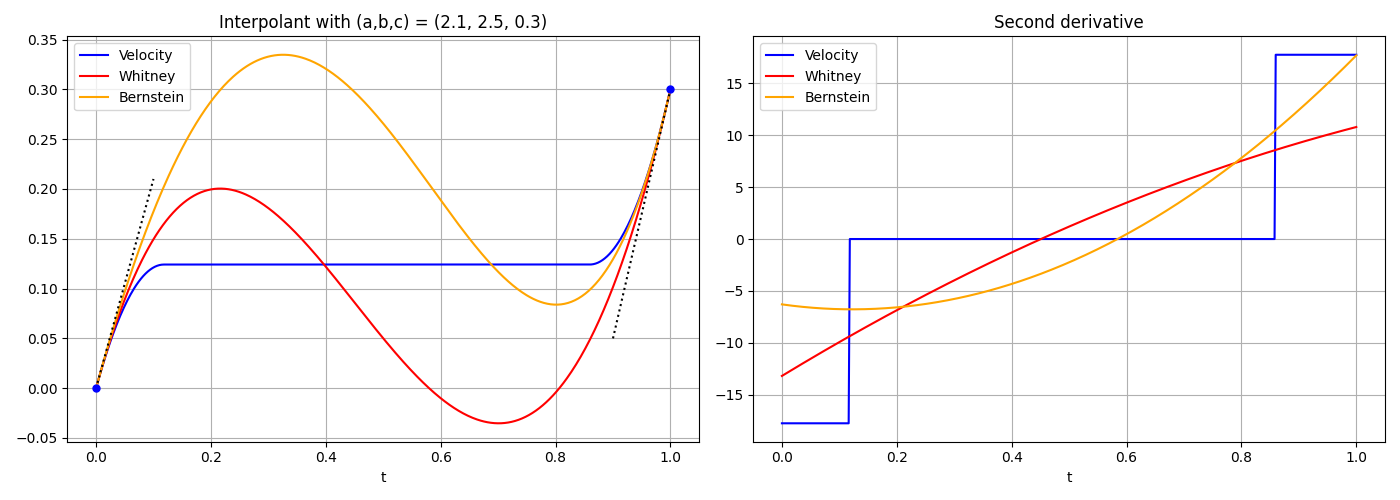}
    \caption{Method comparison when $c \ll \frac{a+b}{2}$}
    \label{fig:small c, no bezier}
\end{figure}

In Figure \ref{fig:random}, we apply the two-point velocity-based method to treat monotone Hermite interpolation problem with ten random nodes. This is based on Theorem \ref{thm:explicit complete} and the discussion in Section \ref{sect:reduction}.

\begin{figure}[htbp]
    \centering
    \includegraphics[width=\linewidth]{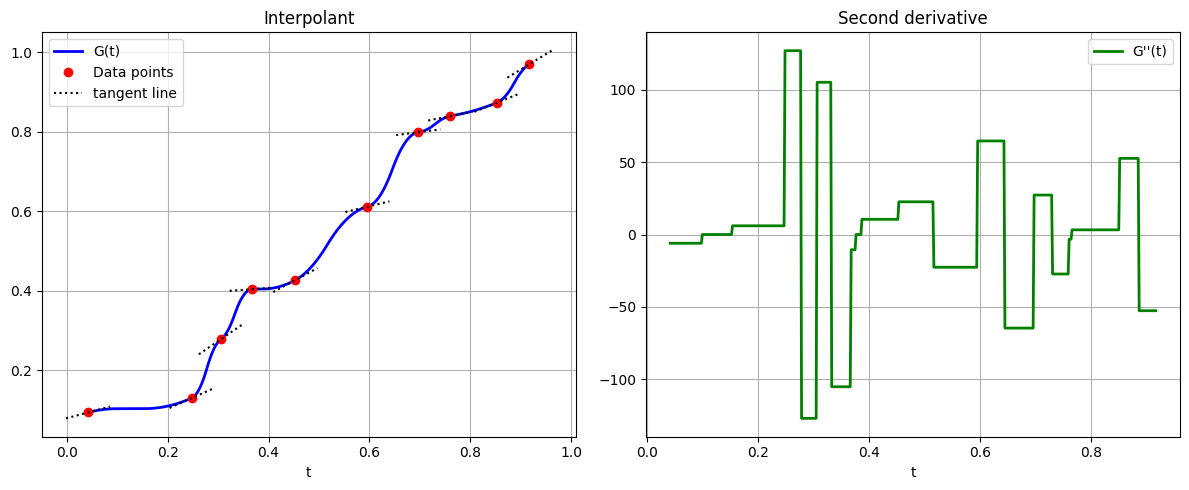}
    \caption{Interpolating five random data points with given slopes}
    \label{fig:random}
\end{figure}

\newpage
\bibliographystyle{plain}
\bibliography{monotone, whitney}

\end{document}